\documentclass[12pt,a4paper]{amsart}

\usepackage{verbatim,amssymb,hyperref,showkeys}

  \theoremstyle{plain}
    \newtheorem{thm}{Theorem}[section]
    \newtheorem{prop}[thm]{Proposition}

  \theoremstyle{definition}

  \theoremstyle{remark}
    \newtheorem*{rem}{Remark}
    
    \newtheorem*{eg*}{Example}

\newcommand{\sym}[1]{\mathfrak{S}_{#1}}
\newcommand{\Lie}{\operatorname{Lie}}

\newcommand{\func}{\operatorname}

\newcommand{\soc}{\func{soc}}
\newcommand{\up}[1]{\text{$\uparrow$}^{#1}}
\newcommand{\down}[1]{\text{$\downarrow$}_{#1}}
\newcommand{\GL}{\mathrm{GL}}

\begin{document}

\baselineskip=18pt

\parindent0pt

\date{December 2010}

\title[Non-projective part of Lie module]{The non-projective part of \\ the Lie module for the symmetric group}

\author{Karin Erdmann}
\address[K. Erdmann]{Mathematical Institute, 24--29 St Giles', Oxford, OX1 3LB, United Kingdom.}
\email{erdmann@maths.ox.ac.uk}

\author{Kai Meng Tan}
\address[K. M. Tan]{Department of Mathematics, National University of Singapore, Block S17, 10 Lower Kent Ridge Road, Singapore 119076.}
\email{tankm@nus.edu.sg}

\thanks{Supported by EPSRC grant EP/G025487/1 and Singapore MOE Academic Research Fund R-146-000-135-112.}
\thanks{Most of the work appearing here was done during the second author's visit to the Mathematical Institute, Oxford, UK, in May 2010.  He thanks the first author for her invitation and the Institute for its hospitality.}

\begin{abstract}
The Lie module of the group algebra $F\sym{n}$ of the
 symmetric group  is known to be not projective if and only
if the characteristic $p$ of $F$ divides $n$.
We show that in this case its  non-projective summands belong
to the principal block of $F\sym{n}$.

Let $V$ be a vector space of dimension $m$ over $F$, and let $L^n(V)$
be the $n$-th homogeneous part of the free Lie algebra on $V$; this is a polynomial representation of $\GL_m(F)$ of degree $n$, or equivalently, a module of the Schur algebra $S(m,n)$.
Our result implies that, when $m \geq n$, every summand of $L^n(V)$ which
is not a tilting module belongs to the
principal block of $S(m,n)$, by which
we mean the block containing the $n$-th symmetric power of $V$.
\end{abstract}

\maketitle

\section{Introduction}

The Lie module of the symmetric group $\sym{n}$ appears in many contexts; in particular it is
closely related to the free Lie algebra.
It may be defined as
the right ideal of the group algebra $F\sym{n}$, generated by the
`Dynkin-Specht-Wever element'
$$\omega_n:= (1-c_n)(1-c_{n-1})\dotsm (1-c_2)$$
where $c_k$ is the $k$-cycle $(1 \ k \ k-1 \ \dotsc \ 2)$.  We write
$\Lie(n) = \omega_nF\sym{n}$ for this module.

One motivation comes from the work of Selick and Wu \cite{SW}. They reduce
the problem of finding  natural homotopy decompositions of the loop suspension of a $p$-torsion suspension to an algebraic question, and
in this context it is important to know a maximal projective submodule
of $\Lie(n)$ when the field has characteristic $p$.
The Lie module also occurs naturally as homology of configuration spaces, and
in other contexts.
Moreover the representation theory of symmetric groups
over prime characteristic is difficult and many basic questions are
open; naturally occurring representations are therefore of interest and may give new understanding.

We recall that $\Lie(n)$ has dimension $(n-1)!$ and its restriction to
$\sym{n-1}$ is free
of rank 1 (see for example \cite{W}, or the discussion in \cite{SW2}).
It is well-known that
$\omega_n^2 = n\omega_n$ (for example it follows from Theorem 5.16 of \cite{MKS}),
 so if $n$ is non-zero in $F$ then $\Lie(n)$ is a direct
summand of the group algebra and hence is projective.
We are interested in this module when $F$ has prime characteristic $p$ and when
$p$ divides $n$. In this case, $\Lie(n)$ always has non-projective summands;
for example it is not projective on restriction to a Sylow subgroup
(which one can see from its dimension). Its
module structure is  not well-understood in general, and the main problem is
to understand the non-projective summands.

Any $F\sym{n}$-module is a direct sum of block components.
Here we show that any non-projective summand of the Lie module belongs
to the principal block.
The Lie module $\Lie(n)$ is the image under the Schur functor of
the $n$-th homogeneous part $L^n(V)$ of the free Lie algebra on $V$ where
$V$ is the natural $m$-dimensional module of $\GL_m(F)$ with $m \geq n$.
Our result implies that any summand of $L^n(V)$ which is not a
summand of $V^{\otimes n}$ belongs to the principal block of the Schur algebra, that is
the block containing the $n$-th symmetric power of $V$.

The organisation of the paper is as follows:  we provide the relevant background in the next section, and prove our results in section \ref{S:main}.

\section{Preliminaries}

In this section, we provide the relevant background information.
Throughout, we fix a field $F$ of prime characteristic $p$, and we assume
it to be infinite when we work with the general linear group, or the Schur
algebra.

\subsection{Blocks}
Let $A$ be a finite-dimensional algebra over $F$.  Then $A$ can be uniquely decomposed into a direct sum of indecomposable two-sided ideals, which are called {\em blocks}.
Each block $B$ of $A$ is associated to a unique primitive central idempotent $e_B$ of $A$.  Let $M$ be a nonzero right $A$-module $M$.  We say that $M$ belongs to $B$, or equivalently, $B$ contains $M$, if and only if $Me_B = M$.
Any $A$-module is a direct sum of block components.

Let $G$ be a finite group.
To each block $B$ of the group algebra $FG$, we associate a $p$-subgroup $D$
of $G$, unique up to conjugation, called the {\em defect group} of $B$,
which may be defined as a minimal $p$-subgroup $Q$ of $G$ such that every right $FG$-module $M$ belonging to $B$ is a direct summand of $(M \down{Q}) \up{G}$.

The block containing the trivial $FG$-module $F$ is called the {\em principal block} of $FG$.

\subsection{Symmetric groups}

Let $\sym{n}$ be the symmetric group on $n$ letters.  The blocks of
$F\sym{n}$ are parametrised by the $p$-cores of partitions of $n$.  The Specht module $S^{\lambda}$, where $\lambda$ is a partition of $n$, belongs to the block $B$ labelled by the $p$-core $\kappa$ of $\lambda$, which is the partition
 obtained by removing as many rim $p$-hooks from the Young diagram of $\lambda$
as possible.  If 
$\kappa$ is a partition of $k$, and the defect group of this block can be taken to be a Sylow $p$-subgroup of $\sym{n-k}$ (see \cite{JK}).

Since the Specht module $S^{(n)}$ is isomorphic to the trivial $F\sym{n}$-module $F$, we see that the principal block of $F\sym{n}$ is parametrised by the $p$-core partition $(r)$ where $r$ is the remainder of $n$ when divided by $p$.

\subsection{Schur algebras}

Let $m,n \in \mathbb{Z}^+$.  The set $\Lambda(m,n)$ of the partitions of $n$ with at most $m$ non-zero parts parametrise the simple modules of the Schur algebra $S(m,n)$.  For $\lambda \in \Lambda(m,n)$, the simple module $L(\lambda)$
has highest weight $\lambda$.  For $\lambda \in \Lambda(m,n)$, there is also a distinguished $S(m,n)$-module, called the tilting module $T(\lambda)$.  This is an indecomposable and self-dual module; it has highest weight $\lambda$ and
has $L(\lambda)$ as a composition factor.  If $\lambda$ is $p$-regular and $m\geq n$ then $T(\lambda)$ has a simple head and socle, and is both injective and projective.

When $m \geq n$, the representation theory of $S(m,n)$ and of $F\sym{n}$ are related via the Schur functor $f : \text{\textbf{mod}-}S(m,n) \to \text{\textbf{mod}-}F\sym{n}$.  This is an exact functor, which preserves direct sums,
and sends a simple $S(m,n)$-module to either a simple $F\sym{n}$-module or to zero.
Furthermore, when $T$ is an injective (and projective) tilting module,  $f(T)$ is an injective (and projective) indecomposable $F\sym{n}$-module, and $f(\soc(T)) = \soc(f(T))$.  In particular, the head (and socle) of $T$ remains nonzero under $f$.

The $p$-cores of partitions in $\Lambda(m,n)$ also play an important role in
the parametrization of the blocks of $S(m,n)$ (see \cite{D}).
When $m\geq n$, such a $p$-core partition $\kappa$ labels a unique block of $S(m,n)$. It contains all
simple modules and also all tilting modules whose labels have $p$-core
$\kappa$.  The Schur functor respect the blocks:  it sends modules belonging to the block of $S(m,n)$ labelled by $\kappa$ to modules belonging to the block of $F\sym{n}$ labelled by $\kappa$.

When $m<n$, these $p$-cores do not necessarily label the blocks of $S(m,n)$, but
the following is still true: $L(\lambda)$ and $L(\mu)$ (or equivalently, $T(\lambda)$ and
$T(\mu)$) belong to the same block of $S(m,n)$ only if $\lambda$ and $\mu$ have the same
$p$-core. Therefore the blocks of $S(m,n)$ can be grouped into disjoint classes
where each class is labelled by a $p$-core partition.

In this paper, we define the `principal block' of $S(m,n)$ to be  the block, or
the class of blocks, labelled by the $p$-core partition
$(r)$ where $r$ is the remainder of $n$ when divided by $p$.

\subsection{Modular Lie powers and the Lie module}

Let $V$ be an  $m$-dimensional vector space over $F$, which we view as
the natural right $F\GL_m(F)$-module.  This induces a $\GL_m(F)$-action on $V^{\otimes n}$.  Denote by $L^n(V)$ the $n$-th homogeneous part of the free Lie algebra generated by $V$ (we take the free Lie algebra to be the
Lie subalgebra generated by $V$ of the tensor algebra on $V$).  Then $L^n(V)$ is
 a submodule of $V^{\otimes n}$.  As both $V^{\otimes n}$ and $L^n(V)$ are polynomial representations of $\GL_m(F)$ of degree $n$, they are thus $S(m,n)$-modules (see \cite{G}).

The indecomposable summands of $V^{\otimes n}$ are precisely the tilting modules $T(\lambda)$ for $p$-regular $\lambda \in \Lambda(m,n)$.
When $m\geq n$, the tensor space $V^{\otimes n}$ is therefore
both injective and projective as an $S(m,n)$-module.  It is known that $L^n(V)$ is a direct summand of $V^{\otimes n}$ if  $p \nmid n$.  When $m \geq n$, the Schur functor sends
$V^{\otimes r}$ and $L^n(V)$  to $F\sym{n}$ and $\Lie(n)$ respectively.

\section{Main results} \label{S:main}



\begin{thm} \label{T:sym}
Any non-projective indecomposable summand of $\Lie(n)$ belongs to the principal block of $F\sym{n}$.
\end{thm}



\begin{proof}
Let $M$ be a non-projective indecomposable summand of $\Lie(n)$.  Let $B$ be the block containing $M$, and let $Q$ be the defect group of $B$.
The module $M$ is a direct summand of $(M \down{Q}) \up{\sym{n}}$.  Thus, since $M$ is non-projective, so is $M \down{Q}$.
On the other hand, $\Lie(n)\down{\sym{n-1}}$ is projective
(see for example \cite{W}, or the discussion in \cite{SW2}),
and thus so is $M \down{\sym{n-1}}$.
Therefore $Q$ cannot be a subgroup of $\sym{n-1}$.

Let $\kappa$ be the $p$-core partition labelling $B$, and suppose that $\kappa$ is a partition of $k$. Then $Q$ can be taken to be a Sylow $p$-subgroup of $\sym{n-k}$.  Since $Q$ is not a subgroup of $\sym{n-1}$, we must have $k=0$, and hence $\kappa = \emptyset$.  This proves that $B$ is the principal block.
\end{proof}

\begin{thm} \label{T:Schur}
Let $V$ be the natural right $F\GL_m(F)$-module. Then any indecomposable summand of $L^n(V)$, which is not of the form $T(\lambda)$ for some $p$-regular
$\lambda \in \Lambda(m,n)$, belongs to the principal block of $S(m,n)$.
\end{thm}

This follows from Theorem \ref{T:sym} and the following proposition.

\begin{prop} \label{P}
Let $m,n \in \mathbb{Z}^+$ with $m \geq n$.  Let $M$ be a finite direct sum of injective tilting modules of $S(m,n)$. Let $N$ be a submodule of $M$ such that $f(N)$ is injective.  Then $N$ is a direct sum of injective tilting modules of $S(m,n)$.
\end{prop}

\begin{proof}
Let $I(N)$ be the injective hull of $N$.  Then $\soc(N) = \soc(I(N))$.  Since $N \subseteq M$, and $M$ is a direct sum of injective tilting modules, we see that $I(N)$ is a direct sum of injective tilting modules, so that $f(\soc(I(N))) = \soc(f(I(N)))$.  Thus,
\begin{equation*}
f(\soc(I(N))) = f(\soc(N)) \subseteq \soc(f(N)) \subseteq \soc(f(I(N))) = f(\soc(I(N))),
\end{equation*}
so that we must have equality throughout.  This yields in particular $\soc(f(N)) =\soc(f(I(N)))$, so that $f(N) = f(I(N))$ since both $f(N)$ and $f(I(N))$ are injective.  As $f$ is exact, we have $f(I(N)/N) \cong f(I(N))/ f(N) = 0$.  Since all the composition factors of the head of $I(N)$ remain nonzero under the Schur functor $f$, we must have $N= I(N)$.
\end{proof}

\begin{proof}[Proof of Theorem \ref{T:Schur}]
Suppose first that $m\geq n$.
Let $N$ be an indecomposable summand of $L^n(V)$ belonging to a non-principal
block of $S(m,n)$.  Then since the Schur functor $f$ preserves direct sums
and respects blocks, $f(N)$ is a summand of $\Lie(n)$
belonging to a non-principal block of $\sym{n}$,
so that by Theorem \ref{T:sym}, $f(N)$ is projective, and hence also is injective.
Since $N$ is a submodule of $V^{\otimes n}$, and $V^{\otimes n}$ is a direct sum of injective tilting modules, it follows from by Proposition \ref{P} that $N$ is a direct sum of injective tilting modules.

Now assume that $m<n$, and let $W$ be the natural $n$-dimensional right $F\GL_n(F)$-module. The truncation functor $d_{n,m}: \text{\textbf{ mod}-}S(n,n) \to \text{\textbf{mod}-}S(m,n)$ (see \cite[\S6.5]{G}) is exact, preserves direct sums, and sends $W$, $W^{\otimes n}$ and $L^n(W)$ to $V$, $V^{\otimes n}$ and $L^n(V)$ respectively. In addition, for a partition $\lambda$ of $n$, it sends $L(\lambda)$ and $T(\lambda)$ of $S(n,n)$ to $L(\lambda)$ and $T(\lambda)$ of $S(m,n)$ if $\lambda$ has at most $m$ non-zero parts, and to zero otherwise (see \cite[\S6.6]{G} and \cite[\S1.7 Proposition]{E}).  In particular, each summand of $L^n(W)$ belonging to the principal block is
sent by $d_{n,m}$ to a summand of $L^n(V)$ belonging to the principal `block' of $S(m,n)$. As each of the summands of $L^n(W)$ belonging to a non-principal block of $S(n,n)$ is a tilting module, it is either sent to zero,
or sent to the tilting module with the same label. This shows that each indecomposable summand of $L^n(V)$ belonging to a non-principal `block' of $S(m,n)$
is a tilting module, labelled by a $p$-regular partition of $n$.
\end{proof}

\end{document}